\documentclass[10pt, a4paper, oneside, reqno]{amsart}
                                          
\usepackage{amsmath,amsthm,amssymb}
\usepackage{tensor,slashed}
\usepackage{physics}
\usepackage[all,cmtip]{xy}
\usepackage[T1]{fontenc}

\usepackage{lmodern}
\usepackage[utf8]{inputenc}				%

\usepackage[english]{babel}

\usepackage[dvipsnames]{xcolor}

\usepackage[normalem]{ulem}
\usepackage{scalerel,stackengine}

\usepackage{hyperref}
\usepackage{JK}

\usepackage{microtype}

\DeclareRobustCommand{\SkipTocEntry}[5]{}

\begin{document}

\begin{abstract}
We consider the assembly map for principal bundles with fiber a countable discrete group. We obtain an index-theoretic interpretation of this homomorphism by providing a tensor-product presentation for the module of sections associated to the Miščenko line bundle. In addition, we give a proof of Atiyah's $L^2$-index theorem in the general context of flat bundles of finitely generated projective Hilbert $C^*$-modules over compact Hausdorff spaces. We thereby also reestablish that the surjectivity of the Baum-Connes assembly map implies the Kadison-Kaplansky idempotent conjecture in the torsion-free case. Our approach does not rely on geometric $K$-homology but rather on an explicit construction of Alexander-Spanier cohomology classes coming from a Chern character for tracial function algebras.
\end{abstract}

\title[Index theory on the Mi\v{s}\v{c}enko bundle]{Index theory on the Mi\v{s}\v{c}enko bundle}

\author{Jens Kaad \and Valerio Proietti}

\address{Department of Mathematics and Computer Science, Syddansk Universitet, Campusvej 55, 5230 Odense M, Denmark}
\email{kaad@imada.sdu.dk}

\address{Research Center for Operator Algebras, East China Normal University, 3663 ZhongShan North Road, Putuo District, Shanghai 200062, China}
\email{proiettivalerio@math.ecnu.edu.cn}



\maketitle

\setcounter{tocdepth}{2}
\tableofcontents



\addtocontents{toc}{\SkipTocEntry}\section*{Introduction}


One of the important applications of the Baum-Connes conjecture is to the Kadison-Kaplansky idempotent conjecture, which asserts that the reduced group $C^*$-algebra of a countable discrete torsion-free group contains no non-trivial idempotents. Indeed, it holds that surjectivity of the Baum-Connes assembly map implies the idempotent conjecture \cite{BaCo:GLF,valette:bc}. The proof of this implication uses two main ingredients, namely the computation of analytic $K$-homology for finite CW complexes using geometric $K$-homology \cite{bhs:gkakhomology}, in combination with Atiyah's $L^2$-index theorem \cite{atiyah:ltwo}. In this paper a proof of this implication is provided which avoids the description of analytic $K$-homology using Baum-Douglas geometric $K$-cycles.

The aim of the present paper is thus twofold: on one hand we wish to clarify the index-theoretic interpretation of the assembly map for torsion-free groups, on the other we intend to show Atiyah's $L^2$-index theorem (Theorem \ref{t:gindtriv_intro} below) by means of a topological argument, involving nothing more than $K$-theory and the Chern character with values in Alexander-Spanier cohomology of the base space.

It is well-known, though maybe not well-documented, that the Miščenko-Fomenko index map coincides with the Baum-Connes assembly map, once the relevant $K$-homology groups are identified (Corollary \ref{cor:cdiag} below). To our knowledge, the only published proof of this result is in \cite{markus:ass}. The argument there makes use of propositions on fixed-point algebras from \cite{siegfried:module}, combined with a clever argument involving dual coactions on crossed products. 

The main obstacle towards a more direct proof, we think, seems to be the usual description of the module of sections associated to the Miščenko line bundle, which is not at first glance amenable to be analyzed through the standard tools of $\KKK$-theory, e.g., the Kasparov product and the descent homomorphism.

In the first part of this paper we provide a structure theorem for the Miščenko module in terms of tensor products and crossed products of Hilbert $C^*$-modules (Theorem \ref{thm:tprdM} below). This presentation is compatible with the basic functorial properties of $\KKK$-theory, and it allows for a different proof of the main theorem in \cite{markus:ass}. It turns out that this structure theorem can also be derived from more general results on weakly proper actions, which can be found in \cite{siegfried:universal,siegfried:module}. 

The second part of this paper is devoted to an index theorem for flat bundles of finitely generated projective modules over a unital $C^*$-algebra equipped with a faithful tracial state. The $L^2$-index theorem of Atiyah is a special case of this index theorem. The proof presented here works in the setting of flat bundles over any second-countable compact Hausdorff space. The extra assumption of second-countability is there for technical reasons relating to the interior Kasparov product. We conjecture that the theorem remains true for general compact Hausdorff spaces and that a proof can be given using the notion of $\KKK$-theory for non-separable $C^*$-algebras due to Skandalis \cite{AAS:BLU,Ska:OES}.

All other proofs of the $L^2$-index theorem known to the authors are set in a smooth setting, (\cite{skandalis:flat,skandalis:rho,atiyah:ltwo,mislin:ltwo,lucsch:vtt,schick:ltwo}), where the manifold structure is used to get a description of the Chern character in terms of connections (i.e., Chern-Weil theory). This approach can then be used in combination with the Baum-Douglas picture of $K$-homology (see \cite{baumdoug:khom,bhs:gkakhomology}) since this describes the entire analytic $K$-homology using data of geometric origin. Remark that geometric $K$-homology is only known to be isomorphic to analytic $K$-homology for (locally) finite CW complexes. 

In particular, the use of Baum-Douglas geometric $K$-cycles has been crucial in deriving the Kadison-Kaplansky idempotent conjecture (in the context of torsion-free groups satisfying the Baum-Connes conjecture) as a corollary of the $L^2$-index theorem (see for example \cite[Section 6.3]{valette:bc}).

Our main motivation for writing this paper was to provide a ``self-contained and topological'' proof of the fact that the surjectivity of the Baum-Connes assembly map implies the Kadison-Kaplansky idempotent conjecture for countable discrete torsion-free groups. By this we mean that the proof should not rely on the geometric $K$-homology description of analytic $K$-homology, but rather be based on the original Kasparov picture of $\KKK$-theory, together with topological considerations. In particular, our proof does not involve differential geometric entities such as connections and differential operators on manifolds.

Instead, our goals are achieved by virtue of Alexander-Spanier cohomology, whose definition incorporates a ``diagonal localization'' feature which we exploit to compute the index pairing with the Miščenko line bundle. An important ingredient is therefore supplied by the explicit description of the Chern character for compact Hausdorff spaces with values in Alexander-Spanier cohomology invented by Gorokhovsky \cite{goro:chern}.

%
%

\addtocontents{toc}{\SkipTocEntry}\section*{Acknowledgements}

We would like to thank Magnus Goffeng for proposing the use of $\text{II}_1$-factors, and Ryszard Nest for suggesting to look into localization at the diagonal in cohomology. Thanks also to Philipp Schmitt for our discussion on Proposition \ref{l:cover}. In addition, this work benefited from conversations with Paolo Antonini, Sara Azzali, and Georges Skandalis. We are grateful to these three authors as well.

The first author was partially supported by the DFF-Research Project 2 ``Automorphisms and Invariants of Operator Algebras'', no. 7014-00145B and by the Villum Foundation (grant 7423). The second author was supported by the Danish National Research Foundation through the Centre for Symmetry and Deformation (DNRF92), and by the Science and Technology Commission of Shanghai Municipality (STCSM), grant no. 13dz2260400.

\section{Preliminaries and main results}

Let $G$ be a countable discrete group and let us fix a second-countable, locally compact, Hausdorff space $\tilde{X}$, equipped with a free and proper action of $G$. We will moreover assume that the quotient space $\tilde{X}/G=X$ is compact and we note that the quotient map $p : \tilde{X}\to X$ forms a principal $G$-bundle. The action of $G$ on $\tilde{X}$ induces an action on the $C^*$-algebra $C_0(\tilde{X})$, which we denote by 
\[
\alpha \colon G\to \mathrm{Aut}(C_0(\tilde{X})) .
\]
We use the convention that $G$ acts on $\tilde{X}$ \emph{from the right} and the induced action on the $C^*$-algebra is therefore given by $\al_g(f)(\tilde{x}) = f( \tilde{x} \cd g)$.
 
Let us turn to the description of the Baum-Connes assembly map 
\[
\mu_{\tilde{X}} : \KK[*][G]{C_0(\tilde{X})}{\Cc} \to \KK[*][]{\Cc}{C^*_r(G)} 
\]
associated to $p : \tilde{X} \to X$. The left-hand side is the $G$-equivariant $K$-homology of the $G$-space $\tilde{X}$ whereas the right-hand side is the $K$-theory of the reduced group $C^*$-algebra $C_r^*(G)$.

First of all we recall the construction of Rieffel's imprimitivity bimodule. We merely sketch the proof, mostly to set up notational conventions, and refer the reader to \cite{rieffel:morita,rieffel:proper} for more details (and more general results).

\begin{proposition}\label{prop:rieffelmorita}
There exists a $C^*$-correspondence $Y$, implementing a strong Morita equivalence between the reduced crossed product $C_0(\tilde{X})\rtimes_r G$ and $C(X)$.
\end{proposition}
\begin{proof}[Sketch of proof]
The $C^*$-correspondence $Y$ is defined as a completion of $C_c(\tilde{X})$. The unital $C^*$-algebra $C(X)$ acts from the right as bounded continuous functions using the pullback along the quotient map $p : \tilde{X} \to X$. The full $C(X)$-valued inner product on $C_c(\tilde{X})$ is defined as 
\begin{equation}\label{eq:innprodL}
\bra{\xi}\ket{\zeta}(x) =\sum_{p(y) = x} (\bar{\xi} \cd \zeta)(y),
\end{equation}
where $\xi, \zeta \in C_c(\tilde{X})$ and $x \in X$. Define $Y$ to be the completion of $C_c(\tilde{X})$ with respect to the induced norm. The left action on $Y$ of the reduced crossed product is given by
\begin{equation}\label{eq:leftactR}
f \cd \xi =\sum_{g\in G} f(g)\alpha_g(\xi),
\end{equation}
where $f\in C_c(G,C_0(\tilde{X}))$ and $\xi\in C_c(\tilde{X})$. The assignment
\begin{equation}\label{eq:innprodR}
\Phi(\Theta_{\xi,\zeta})(g)=\xi\alpha_g(\bar{\zeta})
\end{equation}
mapping from rank-one operators (i.e., $\Theta_{\xi,\zeta}(\upsilon)=\xi\bra{\zeta}\ket{\upsilon}$ with $\xi,\zeta \in C_c( \tilde{X})$) to $C_c(G,C_0(\tilde{X}))$, extends to a $\ast$-isomorphism $\Phi \colon \mathcal{K}(Y)\to C_0(\tilde{X})\rtimes_r G$ from the compact operators of the Hilbert $C^*$-module $Y$ to the reduced crossed product.
\end{proof}

\begin{remark}\label{r:cross}
It follows from \cite[Proposition 2.2]{claire:amenexact} that our action of $G$ is amenable and then by \cite[Theorem 5.3]{claire:amenexact} that the \emph{full} crossed product $C_0(\tilde{X})\rtimes G$ is isomorphic to the \emph{reduced} crossed product $C_0(\tilde{X})\rtimes_r G$. In particular, any covariant pair of representations for $C_0(\tilde{X})$ and $G$ gives rise to a representation of the (reduced) crossed product $C_0(\tilde{X})\rtimes_r G$, namely the integrated form.
\end{remark}


The Baum-Connes assembly map $\mu_{\tilde{X}}$ is defined as the composition of the following maps: 
\[
\xymatrix{\KK[*][G]{C_0(\tilde{X})}{\Cc)} \ar[r]^-{\jmath^G_r} & \KK{C_0(\tilde{X})\rtimes_r G}{C^*_r(G)} \ar[d]^-{[Y^*]\hot_{C_0(\tilde{X})\rtimes_r G}-}\\
& \KK{C(X)}{C^*_r(G)} \ar[r]^-{\iota^*} & \KK{\Cc}{C^*_r(G)} , }
\]
where we specify that
\begin{itemize}
\item the upper horizontal map is the reduced version of Kasparov's descent homomorphism (\cite[page 172]{kas:descent});
\item the vertical map is given by interior Kasparov product with the class $[Y^*]\in \KK[0][]{C(X)}{C_0(\tilde{X})\rtimes_r G}$, induced by the dual of Rieffel's imprimitivity bimodule;
\item the lower horizontal map is the pullback along the inclusion $\iota\colon\Cc\hookrightarrow C(X)$.
\end{itemize}

\begin{remark}
The Baum-Connes assembly map is defined more generally for proper actions of $G$ with compact quotient \cite{BCH:class}. In this paper we focus on free and proper actions since we are interested in the link to the Miščenko-Fomenko index map and the Kadison-Kaplansky idempotent conjecture. 
\end{remark}

We now turn to the description of the Miščenko-Fomenko index map 
\[
\eta_{\tilde{X}} : \KK{C(X)}{\Cc} \to \KK{\Cc}{C^*_r(G)}.
\]
This homomorphism is defined as the composition of the following maps:
\[
\xymatrix{\KK{C(X)}{\Cc} \ar[r]^-{\tau_{C^*_r(G)}} & \KK{C(X)\,\hot\, C^*_r(G)}{C^*_r(G)} \ar[d]^-{[M]\hot_{C(X)\hatotimes C^*_r(G)}-}\\ & \KK{\Cc}{C^*_r(G)} .} 
\]
We specify that the homomorphism $\tau_{C^*_r(G)}$ is defined on Kasparov modules as
\[
(E, F) \mapsto (E\hot C^*_r(G), F\hot 1),
\]
using the exterior tensor product of $C^*$-correspondences, see \cite[Section 17.8.5]{black:kth}. The second homomorphism is given by interior Kasparov product with a class $[M]\in \KK[0][]{\Cc}{C(X)\,\hot\, C^*_r(G)}$, induced by a finitely generated projective Hilbert $C^*$-module $M$. More precisely, $M$ can be identified with the module of sections associated to the \emph{Miščenko line bundle}, i.e., the hermitian bundle of $C^*$-algebras obtained from the associated bundle construction 
\[
\tilde{X}\times_G C^*_r(G)\to X,
\]
where $G$ acts diagonally, acting on the reduced group $C^*$-algebra via the left regular representation \cite{mf:bundle}.

The link between the Baum-Connes assembly map and the Miščenko-Fomenko index map is furnished by the \emph{dual Green-Julg isomorphism}:
\[
J_{\tilde{X}} : \KK[*][G]{C_0(\tilde{X})}{\Cc} \to \KK{C(X)}{\Cc},
\]
which we will describe in details in Section \ref{s:thmA}.

Let us recall that any $G$-$C^*$-correspondence $E$ between two $G$-$C^*$-algebras $A$ and $B$ gives rise to a reduced crossed product $E \rtimes_r G$, which is then a $C^*$-correspondence between the corresponding reduced crossed product $C^*$-algebras, thus from $A \rtimes_r G$ to $B \rtimes_r G$, see \cite[page 170-171]{kas:descent}, \cite{com:cross}, and Section \ref{s:thmA} for more details.

Our first result is the following:
\begin{theoremA}\label{thm:tprdM}
There exists a $G$-$C^*$-correspondence $Z$ from $C_0(\tilde{X})$ to $C(X)$, inducing a class $[Z]\in \KK[0][G]{C_0(\tilde{X})}{C(X)}$, such that the inverse of the dual Green-Julg isomorphism is given by
\[
J_{\tilde{X}}^{-1}=[Z]\hot_{C(X)}- \colon \KK{C(X)}{\Cc}\to \KK[*][G]{C_0(\tilde{X})}{\Cc}.
\]
Moreover, there is an isomorphism of Hilbert $C^*$-modules over $C(X) \hot C_r^*(G)$:
\[
M\cong Y^*\,\hot_{C_0(\tilde{X})\rtimes_r G}\, Z\rtimes_r G.
\]
In particular, we have the $\KKK$-theoretic identity
\[
[M]=\iota^*[Y^*]\,\hot_{C_0(\tilde{X})\rtimes_r G}\,\jmath^G_r[Z]\in \KK[0][]{\Cc}{C(X)\,\hot\, C^*_r(G)}.
\]
\end{theoremA}
\begin{corollary}\label{cor:cdiag}
The following diagram is commutative:
\begin{equation}\label{eq:diag}
\xymatrix{\KK[*][G]{C_0(\tilde{X})}{\Cc} \ar[d]^-{J_{\tilde{X}}} \ar[r]^-{\mu_{\tilde{X}}} & \KK[*][]{\Cc}{C^*_r(G)}\\
\KK[*]{C(X)}{\Cc} . \ar[ur]^-{\eta_{\tilde{X}}}} 
\end{equation}
\end{corollary}
The previous corollary is well-known to experts working on the Baum-Connes conjecture and index theory. It has been proved in \cite{markus:ass} with a different method. 


We now turn to our second result. Let $\phi\colon C^*_r(G)\to \Cc$ denote the canonical tracial state and denote by $\phi_* : K_0(C_r^*(G)) \to \rr$ the induced map on even $K$-theory. Then the homomorphism
\begin{equation}\label{eq:ind}
\phi_*\circ \eta_{\tilde{X}} \colon \KK[0][]{C(X)}{\Cc} \to \rr
\end{equation}
can be interpreted as an index, namely the $L^2$-index of Atiyah \cite{atiyah:ltwo} and, equivalently, the Mi\v{s}\v{c}enko-Fomenko index \cite{mf:bundle}. These identifications are explained in \cite[Theorem 5.15 and Theorem 5.22]{schick:ltwo}. In view of this, we denote the map in \eqref{eq:ind} by
\[
\T{ind}_{C_r^*(G)}\colon \KK[0][]{C(X)}{\Cc} \to \rr.
\]

On the other hand, there is a simple index map
\[
\T{ind} : \KK[0][]{C(X)}{\Cc} \to \zz,
\]
obtained by pairing with the class $[1_{C(X)}]\in K_0(C(X))$, or equivalently by applying the pullback via the unital $*$-homomorphism $\iota\colon\Cc \hookrightarrow C(X)$.

\begin{theoremA}\label{t:gindtriv_intro}
There is an equality of index maps:
\[
\T{ind}_{C_r^*(G)} = \T{ind} : \KK[0][]{C(X)}{\Cc} \to \rr.
\]
In particular, 
\[
\T{ind}_{C_r^*(G)}(x) \in \zz \q \mbox{for all } x \in \KK[0][]{C(X)}{\Cc} .
\]
\end{theoremA}

In fact, we prove a stronger theorem in the context of flat bundles of finitely generated projective Hilbert $C^*$-modules and deduce the above theorem as a special case. We emphasize once more that the space $X$ need not be a CW complex but merely a second-countable, compact Hausdorff space. This provides a generalization of the already known $L^2$-index theorem. 

In \cite{atiyah:ltwo} the previous theorem is proved when $\tilde{X}$ and $X$ are smooth manifolds and without reference to $\KKK$-theory. More precisely, in the smooth setting, the $\KKK$-classes whose index we consider here, are concretely realized in \cite{atiyah:ltwo} as coming from elliptic differential operators, acting on sections of a bundle over $X$, and their lifts to equivariant differential operators acting on the corresponding sections of the pullback bundle. A generalization of Atiyah's $L^2$-index theorem is also proved in \cite{luck:index}, using the universal center-valued trace instead of the standard trace on the group von Neumann algebra. In \cite{schick:ltwo} a generalization of Atiyah's $L^2$-index theorem is provided, which works for flat bundles of finitely generated projective Hilbert $C^*$-modules, but the base space $X$ is still required to be a compact manifold.

The application of the $L^2$-index theorem to the Kadison-Kaplansky idempotent conjecture is based on the following well-known argument \cite[Corollary 6.3.13]{ros:algK}.
\begin{proposition}\label{p:integer}
Let $A$ be a $C^*$-algebra with a unit $1$, and let $\phi$ be a faithful tracial state on $A$. If $\phi_*\colon K_0(A)\to \R$ only takes integer values, then $A$ contains no idempotents other than $0$ and $1$.
\end{proposition}
\begin{proof}
Let $e\in A$ be an idempotent. There is a projection $p\in A$ which is similar to $e$, in particular $\phi(e)=\phi(p)$. Since $1-p\geq 0$, we have $\phi(1)-\phi(p)\geq 0$, and therefore $0\leq \phi(p)=\phi(e)\leq 1$. Now, if $\phi(p)$ is an integer we must have that $\phi(p) \in \{0,1\}$ and therefore since $\phi$ is faithful and $1 \geq p \geq 0$ we conclude that $p \in \{0,1\}$. Thus, since $e$ is similar to $p$, $e$ is equal to either $0$ or $1$. 
\end{proof}

Suppose for the rest of this section that the group $G$ is torsion-free. In this case, Corollary \ref{cor:cdiag} implies that we have a commutative diagram:
\begin{equation}\label{eq:commufull}
\xymatrix{\RK_*^G(C_0(EG),\Cc) \ar[d]^-{J} \ar[r]^-{\mu} & \KK[*][]{\Cc}{C^*_r(G)}\\
\RK_*(C_0(BG),\Cc) \ar[ur]^-{\eta}} .
\end{equation}
Above, $EG\to BG$ is the universal principal $G$-bundle. It is known that the classifying space for proper actions, usually denoted $\underbar{E}G$, coincides (up to equivariant homotopy) with $EG$ when $G$ is torsion-free, because in this case proper actions are automatically free. The groups on the left are $K$-homology groups \emph{with compact support}, and are defined as
\begin{align*}
\RK_*^G(C_0(EG),\Cc)&=\varinjlim_{\tilde{X}\subseteq EG} \KK[*][G]{C_0(\tilde{X})}{\Cc}\\
\RK_*(C_0(BG),\Cc)&=\varinjlim_{X\subseteq BG} \KK[*][]{C(X)}{\Cc},
\end{align*}
where $\tilde{X}$ ranges over locally compact Hausdorff proper $G$-spaces with compact quotient $X$. The various homomorphisms are induced on the direct limits by their ``localized'' counterparts, namely $J_{\tilde{X}}$, $\mu_{\tilde{X}}$ and $\eta_{\tilde{X}}$. 
\medskip

The mentioned application to the Kadison-Kaplansky conjecture is now an immediate consequence of the commutative diagram in \eqref{eq:commufull}, Theorem \ref{t:gindtriv_intro}, and Proposition \ref{p:integer}.

\begin{corollary}
Suppose that $G$ is a discrete countable torsion-free group such that the assembly map $\mu : \RK_*^G(C_0(EG),\Cc)  \to \KK[*][]{\Cc}{C^*_r(G)}$ is surjective. Then every idempotent $e \in C_r^*(G)$ is either equal to $0$ or $1$.
\end{corollary}


\section{\texorpdfstring{Mi\v{s}\v{c}enko}{Mishchenko} module --- Proof of Theorem A}\label{s:thmA}
Recall that $G$ is assumed to be a discrete countable group acting freely, properly, and cocompactly on a second-countable, locally compact Hausdorff space $\tilde{X}$.

We start this section by taking a closer look at the dual Green-Julg isomorphism
\[
J_{\tilde{X}} : \KKK_*^G\big( C_0(\tilde X), \cc \big) \to \KKK_*(C(X),\cc) .
\]
This map is defined as the composition of two isomorphisms:
\begin{align}\label{eq:psiy*}
 \psi & : \KKK_*^G\big( C_0(\tilde X), \cc \big) \to \KKK_*\big( C_0(\tilde X) \rtimes_r G, \cc \big) \q \T{and} \\\notag
 [Y^*]\hot_{C_0(\tilde X) \rtimes_r G}- & : \KKK_*\big( C_0(\tilde X) \rtimes_r G, \cc \big) \to \KKK_*\big( C(X), \cc \big).
\end{align}

The second of these isomorphisms is given by taking interior Kasparov product with the $C^*$-correspondence $Y^*$ from $C(X)$ to $C_0(\tilde X) \rtimes_r G$. This $C^*$-correspondence provides the Morita equivalence between the $C^*$-algebras $C(X)$ and $C_0(\tilde X) \rtimes_r G$ and the homomorphism $[Y^*] \hot_{C_0(\tilde X) \rtimes_r G} -$ is thus an isomorphism with inverse given by the interior Kasparov product with the $C^*$-correspondence $Y$ from $C_0(\tilde X) \rtimes_r G$ to $C(X)$:
\[
[Y] \hot_{C(X)}- : \KKK_*\big( C(X), \cc \big) \to \KKK_*\big( C_0(\tilde X) \rtimes_r G, \cc \big) ,
\]
see Proposition \ref{prop:rieffelmorita}.

We now explain the first of the two isomorphisms in \eqref{eq:psiy*}. 

Suppose that $\mathcal H$ is a countably generated and non-degenerate $G$-$C^*$-correspondence from $C_0(\tilde X)$ to $\cc$. Thus, $\mathcal H$ is a separable Hilbert space equipped with a unitary $G$-action $U\colon G\to \mathcal U(\mathcal{H})$ and a non-degenerate $G$-equivariant $*$-homomorphism $\pi\colon C_0(\tilde{X})\to B(\mathcal{H})$. The left actions of $C_0(\tilde{X})$ and $U$ combine into a non-degenerate left action of the crossed product $C_0(\tilde{X})\rtimes_r G$ on $\mathcal{H}$ in the following way:
\begin{equation}\label{eq:psi}
\wit{\pi}(f \la_g)\xi=(\pi(f)\circ U(g))\xi,\qquad f\in C_0(\tilde{X}) \, , \, \, \xi\in \mathcal{H} .
\end{equation}
Remark that the $*$-homomorphism $\wit \pi : C_0(\tilde X) \rtimes_r G \to B( \mathcal H)$ is indeed well-defined since the reduced crossed product $C_0(\tilde X) \rtimes_r G$ agrees with the full crossed product $C_0(\tilde X) \rtimes G$ in our setting, see Remark \ref{r:cross}. Hence we get a $C^*$-correspondence $\wit{\mathcal H}$ from $C_0(\tilde X) \rtimes_r G$ to $\cc$.

The isomorphism 
\[
\psi : \KKK_*^G\big(C_0(\tilde{X}),\Cc \big) \to \KKK_*\big(C_0(\tilde{X})\rtimes_r G, \Cc \big)
\]
is now defined by the formula $\psi\big( [ \mathcal H,F] \big) = [\wit{\mathcal H},F]$ (the representation is omitted in this notation for Kasparov modules). We need to explain why $\psi$ is an isomorphism and this is most conveniently done by providing an explicit inverse. 
%

Indeed, suppose on the other hand that $\mathcal K$ is a countably generated and non-degenerate $C^*$-correspondence from $C_0(\tilde X) \rtimes_r G$ to $\cc$. We denote the left action by $\rho : C_0(\tilde X) \rtimes_r G \to B(\mathcal K)$. Since $G$ is a countable discrete group we have the inclusion $i : C_0(\tilde X) \to C_0(\tilde X) \rtimes_r G$ and we thus obtain the non-degenerate left action $\widehat \rho = \rho \ci i : C_0(\tilde X) \to B(\mathcal K)$. Moreover, we obtain a group homomorphism $V : G \to \mathcal U(\mathcal K)$ by defining
\[
V(g)(\xi) = \lim_{n \to \infty} \rho(f_n \la_g)(\xi) , \qquad g \in G \, , \, \, \xi \in \mathcal K
\]
for some countable approximate identity $\{ f_n \}_{n\in\N}$ for the $\si$-unital $C^*$-algebra $C_0(\tilde X)$. This data provides us with a $G$-$C^*$-correspondence $\widehat{\mathcal K}$ from $C_0(\tilde X)$ to $\cc$. 

The inverse
\[
\psi^{-1} : \KKK_*\big(C_0(\tilde{X})\rtimes_r G, \Cc \big) \to \KKK_*^G\big(C_0(\tilde{X}),\Cc \big)
\]
is then defined by the formula $\psi^{-1}\big( [\mathcal K,F] \big) = [\widehat{\mathcal K}, F]$.

We are now going to provide a slightly better description of the inverse
\[
J_{\tilde{X}}^{-1} = \psi^{-1} \ci \big( [Y] \hot_{C(X)} - \big) : \KKK_*\big( C(X),\Cc\big) \to \KKK_*^G\big( C_0(\tilde X), \Cc \big)
\]
to the dual Green-Julg isomorphism. As above, using that the left action of $C_0(\tilde X) \rtimes_r G$ on $Y$ is non-degenerate and that the elements in $C(X)$ are $G$-invariant, we obtain a non-degenerate $G$-$C^*$-correspondence
\[
Z = \widehat{Y}
\]
from $C_0(\tilde X)$ to $C(X)$. In fact, recalling that $Y$ is obtained as a completion of $C_c(\tilde X)$, we see that the left action of $C_0(\tilde X)$ on $Z$ comes from the multiplication operation in $C_0(\tilde X)$ and that the $G$-action on $Z$ is induced by
\begin{equation}\label{eq:action}
V(g)(\xi)(x)=\xi(x \cd g), \qquad g \in G \, , \, \, \xi \in C_c(\tilde{X}) \, , \, \, x \in \tilde X .
\end{equation}
When considered as right Hilbert $C^*$-modules over $C(X)$, $Z$ and $Y$ agree. Remark in particular that $C_0(\tilde X)$ acts as compact operators on $Z$ so that $Z$ determines a class $[Z] = [Z,0] \in \KKK_0^G( C_0(\tilde X), C(X))$.

\begin{proposition}\label{p:invjulg}
We have the formula
\[
J_{\tilde{X}}^{-1} = [Z] \hot_{C(X)} - : \KKK_*\big( C(X), \Cc\big) \to \KKK_*^G\big( C_0(\tilde X), \Cc \big)
\]
for the inverse to the dual Green-Julg isomorphism 
\[
J_{\tilde{X}} : \KKK_*^G\big( C_0(\tilde X), \Cc \big)  \to \KKK_*\big( C(X), \Cc\big) .
\]
\end{proposition}
\begin{proof}
This follows immediately by noting that
\[
Z \hot_{C(X)} \mathcal H = \widehat{Y} \hot_{C(X)} \mathcal H = \reallywidehat{ Y \hot_{C(X)} \mathcal{H}},
\]
whenever $\mathcal H$ is a ($\zz/2\zz$-graded) countably generated $C^*$-correspondence from $C(X)$ to $\Cc$. Indeed, for a Kasparov module $(\mathcal H,F_2)$ from $C(X)$ to $\Cc$, the interior Kasparov product
\[
[Y,0] \hot_{C(X)} [ \mathcal H, F_2] \in \KKK_*( C_0(\tilde X) \rtimes_r G,\Cc)
\]
is represented by any Kasparov module $( Y \hot_{C(X)} \mathcal H, F)$ from $C_0(\tilde X) \rtimes_r G$ to $\Cc$, where $F$ is an $F_2$-connection \cite[Definition 18.3.1]{black:kth}. We thus have that
\[
J_{\tilde{X}}^{-1}\big( [\mathcal H, F_2] \big) = \big[ \reallywidehat{Y \hot_{C(X)} \mathcal H}, F\big]
= \big[ Z \hot_{C(X)} \mathcal H, F\big] .
\]
But the $G$-equivariant Kasparov module $(Z \hot_{C(X)} \mathcal H, F)$ from $C_0(\tilde X)$ to $\Cc$ clearly represents the $G$-equivariant interior Kasparov product $[Z,0] \hot_{C(X)} [ \mathcal H,F_2]$ since $F$ is still an $F_2$-connection.
\end{proof}

The proposition above establishes the first half of Theorem \ref{thm:tprdM}. In order to proceed with the second half, we need a more concrete description of the module of sections associated to the Miščenko line bundle. To this end, we use a proposition found in \cite[page 102]{connes:ncg}. We present the details here since they are omitted in \cite{connes:ncg}. 

Let us choose a finite open cover $\{ V_i \}_{i = 1}^N$ of $X$ together with a local trivialization $\phi_i : p^{-1}(V_i) \to V_i \ti G$ for each $i \in \{1,2,\ldots,N\}$. The transition map $\phi_i \ci \phi_j^{-1}$ can then be identified with a continuous map $g_{ij} : V_i \cap V_j \to G$ for each $i,j \in \{1,2,\ldots,N\}$. Notice that since $G$ is discrete each $g_{ij}$ is in fact locally constant and we may thus make sense of the element $\la_{g_{ij}} \in C( V_i \cap V_j, C_r^*(G))$.

\begin{proposition}\label{prop:ConnesM}
The Miščenko module $M$ is the finitely generated projective Hilbert $C^*$-module, described as the completion of $C_c(\tilde{X})$ with respect to the norm induced by the following $C(X)\,\hot \,C^*_r(G)$-valued inner product:
\begin{equation}\label{eq:con1}
\bra{\xi}\ket{\zeta}(t)(x) = \sum_{p(y)=x} \bar{\xi}(y)\zeta(y \cd t),
%
\end{equation}
where $\xi,\zeta\in C_c(\tilde{X})$, $t \in G$, $x \in X$ and $p\colon \tilde{X}\to X$ is the quotient map. The right action of $C(X)\,\hot\, C^*_r(G)$ on $M$ is defined by
\begin{equation}\label{eq:con2}
(\xi \cd f)(y)=\sum_{g \in G} f(g)(p(y)) \cd \xi(y \cd g^{-1}),
\end{equation}
where $\xi\in C_c(\tilde{X})$, $f\in C_c(G, C(X))$ and $y \in \tilde{X}$.
\end{proposition}
\begin{proof}
Choose a partition of unity $\{ \chi_i \}_{i = 1}^N$ such that $\T{supp}(\chi_i) \su V_i$ for all $i \in \{1,2,\ldots,N\}$. For each $i \in \{1,2,\ldots,N\}$ we then define the compactly supported continuous function
\[
\rho_i(y) = \fork{ccc}{ (\chi_i \ci p)(y) & \T{for} & p(y) \in V_i \, \, \T{ and } \, \, \, \phi_i(y) = (p(y), e) \\ 0 & & \T{elsewhere} }
\]
on $\tilde{X}$. The following computation shows that the elements $\{ \sqrt{\rho_i}\}_{i = 1}^N$ form a finite frame for $M$, see \cite[Theorem 4.1]{frank:frames}. Indeed, for each $j \in \{1,2,\ldots,N\}$ and $y \in p^{-1}(V_j)$ with $\phi_j(y) = (x, h)$ (for some $x \in V_j$ and $h \in G$) we have that
\[
\begin{split}
\sqrt{\rho_j} \bra{\sqrt{\rho_j}}\ket{\xi} (y) 
& = \sum_{g \in G} \bra{\sqrt{\rho_j}}\ket{\xi}(g)(x) \cd \sqrt{\rho_j}(y \cd g^{-1}) \\
& = \bra{\sqrt{\rho_j}}\ket{\xi}(h)(x) \cd \sqrt{\chi_j}(x)
= \chi_j(x) \xi( y) .
\end{split}
\]
From this we see that $\xi=\sum_{i=1}^N\sqrt{\rho_i}\bra{\sqrt{\rho_i}}\ket{\xi}$ for all $\xi \in C_c(\tilde{X})$.

But then the projection associated to $M$ takes the form $\bigl(p_{C_r^*(G)}\bigr)_{ij}=\bra{\sqrt{\rho_i}}\ket{\sqrt{\rho_j}}$ and we compute
\[
\begin{split}
\bra{\sqrt{\rho_i}}\ket{\sqrt{\rho_j}}(g)(x) & = 
\sqrt{\rho_i}( \phi_i^{-1}(x,e)) \cd \sqrt{\rho_j}( \phi_i^{-1}(x,e) \cd g) \\
& = \fork{ccc}{
\sqrt{\chi_i \cd \chi_j}(x) & \T{for} & g = g_{ij}(x) \\
0 & \T{for} & g \neq g_{ij}(x) },
\end{split}
\]
whenever $x \in V_i \cap V_j$ and $g \in G$. We thus obtain that 
\[
\bra{\sqrt{\rho_i}}\ket{\sqrt{\rho_j}} = \sqrt{\chi_i\chi_j} \cd \la_{g_{ij}}.
\]
It is now clear that the projection $p_{C_r^*(G)}\in M_N\big(C(X,C_r^*(G))\big)$ describes the module of sections of the hermitian bundle of $C^*$-algebras $\tilde{X} \ti_G C_r^*(G) \to X$.
%
\end{proof}

%

We recall that $Z = \widehat{Y}$ is a non-degenerate $G$-$C^*$-correspondence from $C_0(\tilde{X})$ to $C(X)$. Note that the action of $G$ on $C(X)$ is the trival action and the reduced crossed product $Z \rtimes_r G$ is therefore a $C^*$-correspondence from $C_0(\tilde{X}) \rtimes_r G$ to $C(X) \hot C_r^*(G)$. Before proving the second half of Theorem \ref{thm:tprdM} we recall the formulae for the inner product and the left and right actions on the reduced crossed product $Z \rtimes_r G$. The $C(X) \hot C_r^*(G)$-valued inner product is defined by
\begin{equation}\label{eq:inndesc}
\bra{\xi}\ket{\zeta} = \sum_{g\in G}\bra{\xi(g)}\ket{\zeta(gt)},
\end{equation}
where $\xi,\zeta\in C_c(G,Z)$ and $t\in G$. The right action of $C(X) \hot C_r^*(G)$ is determined by the formula
\[
(\xi \cd f)(t)=\sum_{g\in G}\xi(g) \cd f(g^{-1}t),
\]
where $\xi\in C_c(G,Z),f\in C_c(G, C(X))$ and $t \in G$. The left action of $C_0(\tilde{X}) \rtimes_r G$ on $Z\rtimes_r G$ is determined by 
\begin{equation}\label{eq:leftdesc}
(f\cd \xi)(t)= \sum_{g\in G}f(g) \cd V(g)(\xi(g^{-1}t)),
\end{equation}
where $f\in C_c(G,C_0(\tilde{X})), \xi\in C_c(G,Z)$ and $t \in G$.

\begin{proposition}\label{prop:isoM}
The following map extends to an isomorphism of Hilbert $C^*$-modules,
\begin{align}\label{eq:iso}
& Y^*\,\hot_{C_0(\tilde{X})\rtimes_r G}\, (Z\rtimes_r G) \overset{\Phi}{\longrightarrow} M\\\notag
& \bra{\xi} \otimes \zeta \mapsto \sum_{g\in G} V(g^{-1})\big( \ov{\xi} \cd \zeta(g)\big),
\end{align}
where $\xi\in C_c(\tilde{X}),\zeta\in C_c(G, C_c(\tilde{X}))$. 

In particular, we have the $\KKK$-theoretic identity
\[
[M]=\iota^*[Y^*]\,\hot_{C_0(\tilde{X})\rtimes_r G}\,\jmath^G_r[Z]\in \KK[0][]{\Cc}{C(X)\,\hot\, C^*_r(G)}.
\]
\end{proposition}

\begin{proof}
The result at the level of $\KKK$-theory follows in a straightforward way from the isomorphism in \eqref{eq:iso}, hence we turn to the proof of this isomorphism.

%
It suffices to check that $\Phi : C_c(\tilde{X})^*  \ot C_c(G, C_c(\tilde{X})) \to M$ (defined on the algebraic tensor product over $\Cc$) has dense image and that it preserves the relevant inner products.

The fact that $\Phi$ has dense image follows since for any $\phi \in C_c(\tilde{X}) \su M$ we may find a $\psi \in C_c(\tilde{X})$ such that $\psi \cd \phi = \phi$ (using the pointwise product here). We then have that
\[
\Phi( \bra{ \ov{\psi}} \ot \phi \cd \la_e) = \psi \cd \phi = \phi .
\]

To check that $\Phi$ preserves the inner products we let $\xi_1, \xi_2 \in C_c(\tilde{X})$ and $\ze_1, \ze_2 \in C_c(G, C_c(\tilde{X}))$ and compute, for each $x \in X$ and $t \in G$,
\[
\begin{split}
& \bra{ \bra{\xi_1} \ot \ze_1}\ket{\bra{\xi_2} \ot \ze_2}(t)(x)
= \sum_{g,h \in G} \bra{\ze_1(g)}\ket{\xi_1 \cd V(h)\big(\ov{\xi_2} \cd \ze_2(h^{-1} g t) \big) }(x) \\
& \q = \sum_{g,h \in G} \sum_{p(y) = x} \big( \ov{\ze_1(g)} \cd \xi_1 \big) (y) \cd 
\big( \ov{\xi_2} \cd \ze_2(h^{-1} g t) \big)(y \cd h) ,
\end{split}
\]
where we are using Equation \eqref{eq:innprodL}, \eqref{eq:innprodR}, \eqref{eq:inndesc} and \eqref{eq:leftdesc}.
On the other hand, we have that
\[
\begin{split}
& \bra{ \Phi( \bra{\xi_1} \ot \ze_1)}\ket{\Phi( \bra{\xi_2} \ot \ze_2 ) }(t)(x) \\
& \q = \sum_{s,r \in G} \bra{ V(s^{-1})(\ov{\xi_1} \cd \ze_1(s))}\ket{V(r^{-1})( \ov{\xi_2} \cd \ze_2(r)) }(t)(x) \\
& \q = \sum_{s,r \in G} \sum_{p(z) = x} (\xi_1 \cd \overline{\ze_1(s)})(z \cd s^{-1}) \cd ( \ov{\xi_2} \cd \ze_2(r))(z \cd t \cd r^{-1}) ,
\end{split}
\]
where we are using Equation \eqref{eq:con1}. After a few changes of variables, we obtain that
\[
\bra{ \bra{\xi_1} \ot \ze_1}\ket{ \bra{\xi_2} \ot \ze_2}(t)(x) 
= \bra{ \Phi( \bra{\xi_1} \ot \ze_1)}\ket{ \Phi( \bra{\xi_2} \ot \ze_2 ) }(t)(x)
\]
and this ends the proof of the proposition.
\end{proof}

\begin{remark}
The previous proposition can be interpreted as a particular case of \cite[Prop. 3.6]{siegfried:module}.
\end{remark}

%

\begin{corollary}
The following diagram is commutative:
\begin{equation}\label{eq:ddiag}
\xymatrix{\KKK_*(C_0(\tilde{X}),\Cc) \ar[d]^-{J_{\tilde{X}}} \ar[r]^-{\mu_{\tilde{X}}} & \KKK_*(\Cc,C^*_r(G))\\
\KKK_*(C(X),\Cc). \ar[ur]^-{\eta_{\tilde{X}}}}
\end{equation}
\end{corollary}
\begin{proof}
Suppose $x$ is in $\KK[*][G]{C_0(\tilde{X})}{\Cc}$. Then by Proposition \ref{p:invjulg} there is a $y\in \KK{C(X)}{\Cc}$ with 
\[
x=J_{\tilde{X}}^{-1}(y)=[Z]\,\hot_{C(X)}\, y.
\]
By functoriality of descent \cite[page 172]{kas:descent}, we obtain that
\[
\jmath^G_r(x)=\jmath^G_r([Z]\,\hot_{C(X)}\,y)=\jmath^G_r([Z])\hot_{C(X)\hot C_r^*(G)}\,\jmath^G_r(y).
\]
Note that, since $G$ acts trivially on both $C(X)$ and $\Cc$, we have $C(X)\rtimes_r G\cong C(X)\,\hot\, C^*_r(G)$ and
\[
\jmath^G_r(y)=\tau_{C^*_r(G)}(y).
\]
We thus see that,
\[
\begin{split}
\mu_{\tilde{X}}(x) 
& =\iota^*[Y^*]\hot_{C_0(\tilde{X})\rtimes_r G}\,\jmath^G_r(x) \\ 
& =\iota^*[Y^*]\,\hot_{C_0(\tilde{X})\rtimes_r G}\,\jmath^G_r([Z])\,\hot_{C(X)\,\hot\, C^*_r(G)}\,\tau_{C^*_r(G)}(y).
\end{split}
\]
Applying Proposition \ref{prop:isoM}, the expression above simplifies to
\[
\mu_{\tilde{X}}(x)=[M]\,\hot_{C(X)\hot {C^*_r(G)}}\,\tau_{C^*_r(G)}(y)=\eta_{\tilde{X}}(y).
\]
Hence we have the identity $\mu_{\tilde{X}}(x)=\eta_{\tilde{X}}(J_{\tilde{X}}(x))$ and this proves the corollary.
\end{proof}

%
%
%

\section{Chern characters and flat bundles}

Throughout this section $A$ will be a unital $C^*$-algebra equipped with a faithful tracial state $\phi : A \to \cc$ and $X$ will be a compact Hausdorff space.

We consider the unital $C^*$-algebra $C(X,A) \cong C(X) \hot A$ of continuous $A$-valued maps on $X$.

For every positive integer $n \geq 0$, we will construct an explicit Chern character
\[
\T{Ch}_\phi^{2n} : K_0\big( C(X,A) \big) \to H^{2n}(X,\rr)
\]
with values in the Alexander-Spanier cohomology of $X$. In the case where $A = \cc$, we recover the explicit version of the usual Chern character
\[
\T{Ch}^{2n} : K_0( C(X)) \to H^{2n}(X,\rr)
\]
discovered in \cite{goro:chern}. 

\subsection{Reminders on Alexander-Spanier cohomology }\label{subsec:remASco}

Here is a short summary of how Alexander-Spanier cohomology is defined. For more details, we point the reader to \cite[Chapter 6]{spanier:algtop}.

Let $X$ be a compact Hausdorff space. Let $\T{Cov}(X)$ denote the set of all finite open coverings of $X$, and let $\mathfrak{U}\in \T{Cov}(X)$. For each $k \in \nn \cup \{0\}$, let $\mathfrak{U}^k$ denote the open neighborhood of the diagonal in $X^k$ given by $\cup_{U\in \mathfrak{U}} U^k$, where the superscript $k$ indicates the $k^{\T{th}}$ Cartesian power.  

The real vector space of Alexander-Spanier $k$-cocycles (corresponding to the finite open cover $\mathfrak{U}$) is denoted by $C^k(X, \mathfrak{U})$ and is made of continuous real valued functions on $\mathfrak{U}^{k+1}$. The coboundary map $\partial:C^k(X, \mathfrak{U})\to C^{k+1}(X, \mathfrak{U})$ is defined by the formula  
\begin{equation}\label{eq:deffb}
\partial f(x_0, x_1, \dots, x_{k+1})=\sum_{j=0}^{k+1}(-1)^j f (x_0, \dots, \widehat{x_j}, \dots, x_{k+1}),
\end{equation}
where $f \in C^k(X, \mathfrak{U})$ and the notation $\,\widehat{\cd}\,$ means the term has been omitted.

It can be shown that $\partial^2=0$ and the cohomology of the cochain complex $(C^*(X, \mathfrak{U}), \partial)$ is called the Alexander-Spanier cohomology of the covering $\mathfrak{U}$. It is denoted $H^*(X, \mathfrak{U})$. 

If $\mathfrak{V}$ is a refinement of $\mathfrak{U}$, there is an obvious cochain map (restriction) from $C^*(X, \mathfrak{U})$ to $C^*(X, \mathfrak{V})$, which defines a map $H^*(X, \mathfrak{U}) \to H^*(X, \mathfrak{V})$. 

The (real-valued) Alexander-Spanier cohomology of $X$ is defined as a direct limit over finite open covers:
\[
H^*(X,\rr)=\varinjlim H^*(X, \mathfrak{U}).
\]
These cohomology groups are vector spaces over the real numbers.

\subsection{Construction of the Chern character}
The construction outlined here is entirely based on \cite{goro:chern}. We simply provide a minor generalization of those ideas incorporating the faithful tracial state $\phi : A \to \Cc$.

Let us fix a positive integer $n \geq 0$.
 
Let $p \in M_m( C(X,A))$ be a projection for some positive integer $m \geq 0$ and choose a finite open cover $\G U$ of $X$ such that
\[
\| p( x) - p(x') \| \leq 1/4 \q \forall U \in \G U \, , \, \, x,x' \in U .
\]


We now construct an Alexander-Spanier $2n$-cocycle
\[
\T{Ch}_\phi^{2n}(p) \in C( \G U^{2n+1},\rr) = C^{2n}(X, \G U),
\]
which will represent our Chern character in degree $2n$.

For $n = 0$ we put $\T{Ch}_\phi^0(p)(x) = \phi( p(x))$ for all $x \in X$, where the trace 
\[
\phi : M_m(A) \to \cc
\]
is given by the formula $\phi(a) = \sum_{i = 1}^m \phi(a_{ii})$ for all $a \in M_m(A)$. We remark that $\T{Ch}_\phi^0(p)$ is constant on every $U \in \G U$ since $p(x)$ and $p(x')$ are similar for $x,x' \in U$. In particular, we see that the continuous map $\T{Ch}_\phi^0(p) : X \to \rr$ defines an Alexander-Spanier $0$-cocycle. 

We now consider the case where $n \geq 1$. For every integer $k \geq 1$, we let $\De^k$ denote the $k$-simplex
\[
\De^k = \big\{ (t_1,t_2,\ldots,t_k) \in [0,1]^k \mid \sum_{i = 1}^k t_i \leq 1 \big\} .
\]

Let $x = (x_0, x_1,\ldots,x_{2n}) \in \G U^{2n+1}$. For every $t \in \De^{2n}$, we define
\[
a_p(x,t) = p(x_0) + \sum_{i = 1}^{2n} t_i ( p(x_i) - p(x_0))
\]
and remark that $\| a_p(x,t) - p(x_0) \| \leq 1/4$, in particular we have a well-defined spectral projection
\begin{equation}\label{eq:specproj}
e_p(x,t) = \frac{1}{2 \pi i} \int_{ | \la - 1| = 1/2} ( \la - a_p(x,t))^{-1} \, d\la \in M_m(A),
\end{equation}
where the circle of radius $1/2$ appearing in the formula is given the usual counterclockwise orientation.

The Alexander-Spanier $2n$-cocycle
\[
\T{Ch}_\phi^{2n}(p) \in C( \G U^{2n+1},\rr)
\]
is then defined by the explicit formula
\[
\T{Ch}_\phi^{2n}(p)(x) = \frac{(-1)^n}{ n!} \int_{\De^{2n}} \phi\big( e_p(x,t) d( e_p(x,t) ) \wlw d(e_p(x,t)) \big)
\q x \in \G U^{2n+1}  ,
\]
where the $2n$-simplex $\De^{2n} \subseteq \rr^{2n}$ is given the orientation coming from the form $dt_1 \we dt_2 \wlw dt_{2n}$. Notice that the exterior derivative $d$ appearing in the above expression only differentiates in the direction of the standard simplex $\De^{2n}$.

The proof of the next lemma is almost identical to the proofs of \cite[Lemma~8 and Lemma~9]{goro:chern} and will not be given here.

\begin{lemma}\label{lemma:fchern}
The cochain $\T{Ch}_\phi^{2n}(p) \in C( \G U^{2n+1},\rr)$ is an Alexander-Spanier cocycle and the class $\big[ \T{Ch}_\phi^{2n}(p) \big] \in H^{2n}(X,\rr)$ in Alexander-Spanier cohomology only depends on the class of $p$ in the abelian semigroup $V( C(X,A))$, whose Grothendieck completion gives $K_0(C(X,A))$.
\end{lemma}

It follows from the above lemmas that we have a well-defined map
\[
\T{Ch}_\phi^{2n} : V( C(X,A)) \to H^{2n}(X,\rr)
\]
and it can be verified that this map is a homomorphism, that is
\[
[ \T{Ch}_\phi^{2n}( p \oplus q) ] = [ \T{Ch}_\phi^{2n}(p)] + [\T{Ch}_\phi^{2n}(q)] ,
\]
whenever $p \in M_m( C(X,A))$ and $q \in M_{m'}(C(X,A))$ are projections.

In particular, we have the following:

\begin{dfn}
The Chern character in degree $2n$ associated to the faithful tracial state $\phi : A \to \cc$ and the compact Hausdorff space $X$ is the homomorphism of abelian groups
\[
\T{Ch}_\phi^{2n} : K_0( C(X,A)) \to H^{2n}(X,\rr), \q \T{Ch}_{\phi}( [p] - [q]) = [ \T{Ch}_\phi^{2n}(p)] - [ \T{Ch}_\phi^{2n}(q)]  .
\]
\end{dfn}

\subsection{Multiplicative properties}
We let
\[
\ti : K_0(C(X)) \otimes_\zz K_0(A)  \to K_0\big( C(X,A) \big)
\]
denote the exterior product. Recall that for projections $p \in M_m( C(X))$ and $q \in M_{m'}(A)$, the exterior product
\[
[p] \ti [q] \in K_0\big( C(X,A) \big)
\]
is represented by the projection $p \ot q \in M_{m \cd m'}( C(X,A)) \cong M_{m'}\big( M_m(C(X,A)) \big)$ given by the block-matrix
\[
(p \ot q)_{ij} = p \cd q_{ij} \q i,j \in \{1,\ldots,m'\} .
\]

We recall that the faithful tracial state $\phi : A \to \cc$ induces a homomorphism
\[
\phi_* : K_0(A) \to \rr, \q \phi_*( [p] - [q]) = \phi(p - q)  .
\]

\begin{lemma}\label{l:cheagr}
For every positive integer $n \geq 0$, we have the commutative diagram
\[
\begin{CD}
K_0(C(X)) \otimes_\zz K_0(A) @>{\ti}>> K_0\big( C(X, A) \big) \\
@V{1 \ot \phi_*}VV @V{\T{Ch}_\phi^{2n}}VV \\
K_0(C(X)) \otimes_\zz \rr @>{(\T{Ch}^{2n} \ot 1)}>> H^{2n}(X,\rr)
\end{CD}
\]
\end{lemma}
\begin{proof}
Given $p \in M_m(C(X))$ and $q \in M_{m'}(A)$, the commutativity of the diagram follows from the identity 
\[
[ \Chern^{2n}(p) ] \cd \phi(q) = [ \Chern^{2n}_\phi(p\otimes q) ].
\]
We shall in fact see that this identity holds at the level of cochains. We focus on the case where $n \geq 1$. In this situation, it suffices to show that
\begin{equation}\label{eq:chercomm}
e_{p\otimes q}(x,t)=e_p(x,t)\otimes q,
\end{equation}
for all $x \in \G U^{2n+1}$ and all $t \in \De^{2n}$. Indeed, if Equation \eqref{eq:chercomm} were true, then from Equation \eqref{eq:specproj} we would have that
\begin{align*}
\T{Ch}_\phi^{2n}(p\otimes q)(x) &= \frac{(-1)^n}{ n!} \int_{\De^{2n}} \text{Tr}\big( e_p(x,t) d( e_p(x,t) ) \wlw d(e_p(x,t)) \big)\cd \phi(q)\\
&= \Chern^{2n}(p)(x) \cd \phi(q),
\end{align*}
for all $x \in \G U^{2n + 1}$, where $\T{Tr} : M_m(\cc) \to \cc$ denotes the matrix trace (without normalization). Now, for each $\la \in \cc$ with $|\la - 1| = 1/2$, it is easily verified that
\[
(\lambda- a_{p\otimes q}(x,t))^{-1}=(\lambda- a_p(x,t))^{-1}\otimes q + \frac{1}{\lambda}\otimes (1- q) .
\]
The identity above is exactly what we need, since the function $\frac{1}{\lambda} \otimes (1 -q)$ is analytic on an open set containing $\{\lambda\in\Cc \mid |\lambda-1|\leq 1/2\}$, and therefore its contour integral along the boundary of that disk is zero.
\end{proof}

\begin{prop}\label{prop:injchern}
Suppose that $A$ is a $\T{II}_1$-factor. Then the Chern character
\[
\T{Ch}_\phi : K_0\big( C(X,A) \big) \to \oplus_{n = 0}^\infty H^{2n}(X,\rr) \q \T{Ch}_\phi(x) = \{ \T{Ch}_\phi^{2n}(x) \}
\]
is an isomorphism.
\end{prop}

Without explicit mention of the Chern character, the previous proposition is proved in \cite[Corollary 3]{raeburn:kreal}. It is also proved in \cite[Theorem 5.7]{schick:ltwo} in the smooth setting by using Chern-Weil theory, see also \cite[Diagram 3.7]{skandalis:flat}.

\begin{proof}
Since $A$ is a $\T{II}_1$-factor, it follows as in \cite[III.1.7.9, p.~242]{black:opalg}, that the faithful tracial state $\phi : A \to \cc$ induces an isomorphism $K_0(A) \cong \rr$ of abelian groups. Moreover, since $A$ is a von Neumann algebra, we have that $K_1(A) \cong \{0\}$. Since $\rr$ is torsion-free, the exterior product
\[
\ti : K_0(C(X)) \otimes_\zz K_0(A)  \to K_0(C(X,A))
\]
is an isomorphism by the K\"unneth theorem, see \cite[Proposition 2.11]{scho:kunn}. Therefore it suffices to show that the composition
\[
\T{Ch}_\phi \ci \ti : K_0(C(X)) \otimes_\zz K_0(A) \to \oplus_{n = 0}^\infty H^{2n}(X,\rr)
\]
is an isomorphism. However, by Lemma \ref{l:cheagr} we have that
\[
(\T{Ch}_\phi \ci \ti) = (\T{Ch} \ot 1)\circ(1 \ot \phi_*) ,
\]
where $\T{Ch} : K_0( C(X)) \to \oplus_{n = 0}^\infty H^{2n}(X,\rr)$ is the usual Chern character with values in Alexander-Spanier cohomology. This ends the proof of the proposition since $\T{Ch}$ becomes an isomorphism after tensorizing with $\rr$, see \cite{kar:chern}.
\end{proof}

\subsection{Flat bundles}\label{sec:flat}
We now consider a flat bundle over the compact Hausdorff space $X$ with fiber a finitely generated projective module $q A^m$ over the unital $C^*$-algebra $A$, thus $q \in M_m(A)$ is a projection. We thus fix an open cover $\{ V_i\}_{i = 1}^N$ of our compact Hausdorff space $X$ together with locally constant maps
\[
g_{ij} : V_i \cap V_j \to U(q A^m) \q i,j \in \{1,\ldots,N\} ,
\]
for some fixed $m \in \nn$, where $U(q A^m)$ denotes the group of unitary transformations of the Hilbert $C^*$-module $q A^m$. We identify $U(q A^m)$ with the group of $(m \ti m)$-matrices $u$ satisfying the conditions
\[
q u = u = u q \q \T{and} \q u^* u = u u^* = q .
\]
The locally constant maps are required to satisfy the cocycle condition:
\[
\begin{split}
& g_{ii} = q \q \mbox{and} \\
& g_{ij}(x) \cd g_{jk}(x) = g_{ik}(x) \q \forall x \in V_i \cap V_j \cap V_k,
\end{split}
\]
whenever $i,j,k \in \{1,\ldots,N\}$. 

Let us choose a partition of unity $\{\chi_i \mid i=1,\dots,N\}$ for $X$ with $\T{supp}(\chi_i) \subseteq V_i$ for all $i \in \{1,\ldots,N\}$. Our cocycle then gives rise to a projection $p_A \in M_{N \cd m}(C(X,A)) \cong M_N( M_m(C(X,A)))$ defined as the block-matrix
\begin{equation}\label{eq:pproj}
(p_A)_{ij} = \sqrt{\chi_i\chi_j} \cd g_{ij} \q i,j \in \{1,\ldots,N\}.
\end{equation}

Finally, we have the projection $p \in M_N(C(X,A))$ defined as the matrix
\[
p_{ij} = \sqrt{\chi_i\chi_j}  \q i,j \in \{1,\ldots,N\} .
\]

We are going to prove the following:

\begin{thm}\label{thm:projid}
We have the identity
\[
\T{Ch}^{2n}_\phi( [p_A] ) = \T{Ch}^{2n}_\phi( [q]) = \fork{ccc}{ 0 & \T{for} & n > 0 \\ \, [\phi(q)] & \T{for} & n = 0 } ,
\]
where $[\phi(q)] \in H^0(X,\rr)$ refers to the class in Alexander-Spanier cohomology associated to the constant function on $X$ equal to $\phi(q) \in [0,\infty)$ at every point. 
\end{thm}

When the base space is a compact manifold without boundary, the previous result is proved in \cite[Theorem 5.8]{schick:ltwo} and \cite[Section 4]{skandalis:flat}.

The more general case where the base space is just a compact Hausdorff space requires extra care. We start with a technical lemma.

\begin{lemma}\label{l:cover}
Suppose that $\G K = \{ K_1,K_2,\ldots,K_l\}$ is a finite set of closed subsets of the compact Hausdorff space $X$. Then there exists a finite open cover $\G U$ of $X$ such that the implication
\[
\Big( (U \cap K_i ) \neq \emptyset \, \, \, \forall i \in I \Big) \rar \Big(  \bigcap_{i \in I} K_i \neq \emptyset \Big)
\]
holds for all non-empty subsets $I \su \{1,2,\ldots,l\}$ and all $U \in \G U$.
%
\end{lemma}
\begin{proof}
The case where $\G K$ is empty is trivial, so we suppose that $l = \sharp \G K \geq 1$.

Define the set 
\[
\G A = \big\{ \cap_{i \in I} K_i \mid I \su \{1,2,\ldots,l\} \, , \, \, I \neq \emptyset \big\} \cup \{ \emptyset \}.
\]

Let $n \in \nn$ and suppose that $C_1,C_2,\ldots,C_n \in \G A$ and that $U_1,U_2,\ldots,U_n \su X$ are open subsets such that
\begin{enumerate}
\item $C_1 = U_1 = \emptyset$;
\item $C_j \in \G A \sem \{C_1,C_2,\ldots,C_{j-1} \}$ for all $j \in \{2,3,\ldots,n\}$;
\item it holds for all $K \in \G K$ and all $j \in \{2,3,\ldots,n\}$ that
\[
C_j \cap K = C_j \q \T{or} \q C_j \cap K \in \{C_1,C_2,\ldots,C_{j-1}\};
\]
\item $C_j \cap ( X \sem \cup_{i = 1}^{j-1} U_i) \su U_j$ for all $j \in \{2,3,\ldots,n\}$;
\item the implication
\[
\Big( C_j \cap K \in \{ C_1,C_2,\ldots,C_{j-1} \} \Big) \rar \Big( U_j \cap K = \emptyset \Big)
\]
holds for all $K \in \G K$ and all $j \in \{2,3,\ldots,n\}$.
\end{enumerate}
We remark that $\cup_{i = 1}^j C_i \su \cup_{i = 1}^j U_i$ for all $j \in \{1,2,\ldots,n\}$. Indeed, to see this it suffices to check that $C_j \su \cup_{i = 1}^j U_i$ for $j \in \{2,3,\ldots,n\}$. But this is clear since $C_j \cap (X \setminus \cup_{i = 1}^{j - 1} U_i) \su U_j$ by construction and obviously $C_j \cap ( \cup_{i = 1}^{j - 1} U_i ) \su \cup_{i = 1}^{j-1} U_i$.

Notice that $\G A$ is finite because $\G K$ is finite. If $\G A \sem \{ C_1,C_2,\ldots,C_n\} \neq \emptyset$, we perform the following step: we start by choosing $C_{n + 1} \in \G A \sem \{C_1,C_2,\ldots,C_n\}$ such that it holds for all $K \in \G K$ that
\[
C_{n + 1} \cap K = C_{n + 1} \q \T{or} \q C_{n + 1} \cap K \in \{C_1,C_2,\ldots,C_n\}.
\]
Secondly, we construct an open set $U_{n+1}\subseteq X$ in such a way that $C_1,C_2,\ldots,C_{n+1} \in \G A$ and $U_1,U_2,\ldots,U_{n+1} \su X$ satisfy $(1)-(5)$ from above. The construction of $U_{n+1}$ goes as follows. 
We define
\[
L_{n + 1} = C_{n + 1} \cap \big( X \setminus \cup_{i = 1}^n U_i \big)
\]
and claim that it holds for all $K \in \G K$ that
\[
\big( C_{n + 1} \cap K  \in \{ C_1,C_2,\ldots,C_n\} \big) \rar \big( L_{n + 1} \cap K  = \emptyset \big)
\]
Indeed, if $C_{n + 1} \cap K  \in \{ C_1,C_2,\ldots,C_n\}$, then $C_{n + 1} \cap K \su \cup_{i = 1}^n U_i$ so that
\[
K \cap L_{n + 1} = K \cap C_{n + 1} \cap \big( X \setminus \cup_{i = 1}^n U_i \big) = \emptyset.
\]
Since $X$ is compact Hausdorff (and hence normal), there exists an open subset $U_{n + 1} \su X$ such that
\[
L_{n + 1} \su U_{n + 1}
\]
and such that the implication
\[
\big( C_{n + 1} \cap K \in \{ C_1,C_2,\ldots,C_n\} \big) \rar \big( U_{n + 1} \cap K = \emptyset \big)
\]
holds for all $K \in \G K$.

By iterating the previous step a possibly large but finite number of times, we arrive at families $C_1,C_2,\ldots,C_m \in \G A$ and $U_1,U_2,\ldots,U_m \su X$ satisfying $(1)-(5)$ from above, and such that $\G A = \{ C_1,C_2,\ldots,C_m\}$. We define
\[
U_{m+1}=X\setminus \cup_{i = 1}^l K_i
\]
and claim that $\G U=\{U_i\}_{j=1}^{m+1}$ is the desired open cover. 

First of all, we prove that $\G U$ is indeed a cover. To this end, we just need to show that $\cup_{i = 1}^l K_i \su \cup_{j = 1}^m U_j$, but this is clear since $\cup_{i = 1}^l K_i = \cup_{j = 1}^m C_j \su \cup_{j = 1}^m U_j$.

Now suppose that $I \su \{1,2,\ldots,l\}$ is a non-empty subset, that $U \in \G U$ and that $U \cap K_i \neq \emptyset$ for all $i \in I$. By the definition of $U_1$ and $U_{m+1}$, we must have that $U = U_j$ for some $j\in \{2,\dots,m\}$. By property $(3)$ and $(5)$, it thus holds that $C_j \cap K_i = C_j$ for all $i \in I$. But this implies that
\[
C_j = C_j \cap \big( \cap_{i \in I} K_i \big) \su \cap_{i \in I} K_i
\]
and hence since $C_j \neq \emptyset$ that $\cap_{i \in I} K_i \neq \emptyset$. This proves the lemma.
\end{proof}

For each $i,j \in \{1,\ldots,N\}$ with $i \neq j$ define the closed subset
\[
K_{ij} = \T{supp}( \chi_i) \cap \T{supp}( \chi_j ) \su X
\]
and define 
\[
\G K = \big\{ K_{ij} \mid i,j \in \{1,2,\ldots,N\} \big\}.
\]

Let $\G U$ be a finite open cover of $X$ satisfying the conclusion of Lemma \ref{l:cover}. 
By passing to a refinement we may also arrange that:
\begin{itemize}
\item $\| p(x) - p(x') \| \, \, \T{and} \, \, \| p_A(x) - p_A(x') \| \leq 1/4$ \text{ for all } $U \in \G U$ \text{ and } $x,x'\in U$;
\item for all $U\in\G U$, whenever $V_i\cap V_j\cap U\neq \emptyset$, the map $g_{ij}\colon V_i\cap V_j\cap U \to U(q A^m)$ is constant.
\end{itemize}

The following lemma is exactly what we need to prove Theorem \ref{thm:projid}.

\begin{lemma}\label{l:realtech}
Let $n \geq 1$, let $x = (x_0,x_1,\ldots,x_{2n}) \in \G U^{2n+1}$ and let $t \in \De^{2n}$ be given. We have the identity,
\[
\Big( e_{p_A}(x,t) d\big( e_{p_A}(x,t) \big)^{\we 2n} \Big)_{ij} = e_p(x,t) d\big( e_p(x,t) \big)^{\we 2n} \cd g_{ij}(x_0),
\]
for each $i,j\in\{1,\dots,N\}$ indexing the $(m\times m)$-block matrices.
\end{lemma}
\begin{proof}
Let us choose a $U \in \G U$, such that $x = (x_0,x_1,\ldots,x_{2n}) \in U^{2n+1}$. We remark that $g_{ij}(x_s) = g_{ij}(x_0)$ for all $s \in \{0,1,2,\ldots,2n\}$ and all $i,j \in \{1,2,\ldots,N\}$.

For $\la \in \cc$ with $| \la - 1 | = 1/2$, we define
\[
\begin{split}
\ga_{p_A}(\la,x_0) & = p_A(x_0) / (\la - 1) + (1 - p_A(x_0))/\la \q  \T{and} \\
\de_{p_A}(x,t) & = \sum_{s = 1}^{2n} t_s \cd ( p_A(x_s) - p_A(x_0)) ,
\end{split}
\]
so that $a_{p_A}(x,t) = p_A(x_0) + \de_{p_A}(x,t)$ and
\[
( \la - p_A(x_0) ) \cd \ga_{p_A}(\la,x_0) = 1.
\]
In particular, we have the power-series expansion 
\begin{equation}\label{eq:neumann}
(\la - a_{p_A})^{-1} = \sum_{k = 0}^\infty ( \ga_{p_A}(\la) \cd \de_{p_A} )^k \ga_{p_A}(\la) \q |\la - 1| = 1/2,
\end{equation}
which converges absolutely since $\| \de_{p_A} \| \leq \frac{1}{4}$. Remark that we are suppressing the point $(x,t) \in U^{2n+1} \ti \De^{2n}$ and the point $x_0 \in U$ from the notation (and we will often do so below as well).

%
Notice that the exterior derivative of $(\la - a_{p_A})^{-1}$ (again in the direction of the simplex $\De^{2n}$) can be easily computed:
\[
d \big( (\la - a_{p_A})^{-1} \big) 
= \sum_{s = 1}^{2n} (\la - a_{p_A})^{-1} ( p_A(x_s) - p_A(x_0)) (\la - a_{p_A})^{-1} dt_s .
\]

We thus have that
\[
d( e_{p_A} )
= \frac{1}{2\pi i} \sum_{s = 1}^{2n} \int_{ |\la - 1| = 1/2} 
(\la - a_{p_A})^{-1} ( p_A(x_s) - p_A(x_0)) (\la - a_{p_A})^{-1} \, d\la \, dt_s
\]
and hence that
\[ 
\begin{split}
 e_{p_A} d( e_{p_A} )^{\we 2n}
&  = \sum_{ \si \in S_{2n}} \T{sign}(\si) \cd \frac{1}{(2\pi i)^{2n+1}} 
\int_{ |\la_0 - 1| = 1/2} \ldots \int_{ |\la_{2n} - 1| = 1/2} \\
& \q  (\la_0 - a_{p_A})^{-1} \prod_{r = 1}^{2n} (\la_r - a_{p_A})^{-1} ( p_A(x_{\si(r)}) - p_A(x_0)) (\la_r - a_{p_A})^{-1} \\
& \q \q  d\la_0 \ldots d\la_{2n} \, d t_1 \wlw d t_{2n} ,
\end{split}
\]
where $S_{2n}$ denotes the group of all permutations $\si$ of $2n$ letters.

Since a similar expression holds when $p_A$ is replaced by $p$, we may focus on proving the identity
\begin{align*}
\Big( (\la_0 - a_{p_A})^{-1} &\prod_{r = 1}^{2n} (\la_r - a_{p_A})^{-1} ( p_A(x_{\si(r)}) - p_A(x_0)) (\la_r - a_{p_A})^{-1} \Big)_{ij} \\
= \Big( (\la_0 - a_p)^{-1} & \prod_{r = 1}^{2n} (\la_r - a_p)^{-1} ( p(x_{\si(r)}) - p(x_0)) (\la_r - a_{p})^{-1} \Big)_{ij} \cd g_{ij}(x_0),
\end{align*}
for each fixed $i,j \in \{1,\ldots,N\}$, each permutation $\si \in S_{2n}$ and each $\la_0,\ldots,\la_{2n} \in \cc$ with $|\la_r - 1| = 1/2$ for all $r \in \{0,\ldots,2n\}$.
%

We now apply the power-series expansion from Equation \eqref{eq:neumann} (both for $a_{p_A}$ and $a_p$), so we reduce the lemma to proving that

\begin{multline}\label{eq:firstid}
\Big( ( \ga_{p_A}(\la_0) \cd \de_{p_A} )^{k_0} \ga_{p_A}(\la_0) \\
 \cd \prod_{r = 1}^{2n}( \ga_{p_A}(\la_r) \cd \de_{p_A} )^{k_r} \ga_{p_A}(\la_r)
\cd ( p_A(x_{\si(r)}) - p_A(x_0)) ( \ga_{p_A}(\la_r) \cd \de_{p_A} )^{l_r} \ga_{p_A}(\la_r) \Big)_{ij} 
\end{multline}
is equal to
\begin{multline*}
 \Big( ( \ga_p(\la_0) \cd \de_p)^{k_0} \ga_p(\la_0) \\ 
 \cd \prod_{r = 1}^{2n}( \ga_p(\la_r) \cd \de_p )^{k_r} \ga_p(\la_r)
\cd ( p(x_{\si(r)}) - p(x_0)) ( \ga_p(\la_r) \cd \de_p)^{l_r} \ga_p(\la_r) \Big)_{ij} \cd g_{ij}(x_0),
\end{multline*}
for every $(k_0,\ldots,k_{2n}) \in (\nn \cup \{0\})^{2n+1}$ and every $(l_1,\ldots,l_{2n}) \in (\nn \cup \{0\})^{2n}$.

We are going to reduce our task even further. Thus, let us fix $(k_0,\ldots,k_{2n}) \in (\nn \cup \{0\})^{2n+1}$ and $(l_1,\ldots,l_{2n}) \in (\nn \cup \{0\})^{2n}$. Using indices $\al, \be \in \{1,\ldots,N\}$ for the $(m \ti m)$-blocks we notice that
\begin{align*}
(\ga_{p_A})_{\al \be} \cdot q &= (\ga_p)_{\al \be} \cd g_{\al \be}(x_0)\\
(\de_{p_A})_{\al \be}  &= (\de_p)_{\al \be} \cd g_{\al \be}(x_0) \\
(p_A)_{\al \be}(x_s)   &= p_{\al \be}(x_s) \cd g_{\al \be}(x_0),
\end{align*}
for all $s \in \{0,1,\ldots,2n\}$. Letting
\[
M = 1 + 6n + 2k_0 + 2 \sum_{r = 1}^{2n} (k_r + l_r) 
\]
denote the number of multiplicative factors involved in the operator in Equation \eqref{eq:firstid}, using block-multiplication, again with $(m \ti m)$-blocks, we may rewrite this operator as
\begin{equation*}
\sum_{i_0,\dots, i_{M-2} = 1}^N C_{ii_0i_1 \ldots i_{M-2}j}\cdot (g_{ii_0}g_{i_0i_1} \clc g_{i_{M-3} i_{M-2}}g_{i_{M-2}j})(x_0),
\end{equation*}
where
\begin{multline*}
\sum_{i_0,\dots,i_{M-2} = 1}^N C_{ii_0i_1 \ldots i_{M-2}j} =\Big( ( \ga_p(\la_0) \cd \de_p)^{k_0} \ga_p(\la_0) \\
 \cd \prod_{r = 1}^{2n}( \ga_p(\la_r) \cd \de_p )^{k_r} \ga_p(\la_r)
\cd ( p(x_{\si(r)}) - p(x_0)) ( \ga_p(\la_r) \cd \de_p)^{l_r} \ga_p(\la_r) \Big)_{ij}.
\end{multline*}

Let us now fix indices $i_0, \ldots, i_{M-2} \in \{1,\ldots,N\}$. To ease the notation, we put 
\[
i_{-1} = i \q \T{and} \q i_{M - 1} = j .
\]
It suffices to show that
\[
C_{i_{-1}i_0\ldots i_{M-1}} \cd (g_{i_{-1}i_0} \clc g_{i_{M-2} i_{M-1}})(x_0)
= C_{i_{-1}i_0 \ldots i_{M-1}} \cd g_{i_{-1} i_{M-1}}(x_0).
\]
We claim that if $C_{i_{-1}i_0 \ldots i_{M-1}}$ is nonzero, then
\begin{equation}\label{eq:intersection}
K_{i_\alpha i_{\al + 1}}\cap U\neq\emptyset \qquad \forall \alpha \in \{-1,0,\ldots,M-2\}  
\, \, \T{with } i_\al \neq i_{\al + 1}.
\end{equation}
If Equation \eqref{eq:intersection} holds, then by virtue of Lemma \ref{l:cover} we must have that
\[
\bigcap_{\al = -1}^{M-2} K_{i_\al i_{\al + 1}} \neq \emptyset ,
\]
which in turn means $\bigcap_{\alpha = -1}^{M - 2} V_{i_\alpha i_{\al + 1}}\neq \emptyset$, therefore the cocycle relations hold and 
\[
(g_{i_{-1}i_0} \clc g_{i_{M-2} i_{M-1}})(x_0)=g_{i_{-1} i_{M-1}}(x_0),
\]
which is what we set out to prove.

So let us suppose that $K_{i_\alpha i_{\al + 1}}\cap U=\emptyset$ for some $\al \in \{-1,0,\ldots,M-2\}$ with $i_\al \neq i_{\al + 1}$. We are going to show that $C_{i_{-1} i_0 \ldots i_{M-1}} = 0$. There are three cases: the pair $i_\alpha i_{\al + 1}$ can appear in a term of the form $\ga_p(\lambda,x_0)_{i_\alpha i_{\al + 1}}$, or $(\de_p(x,t))_{i_{\alpha} i_{\al + 1}}$, or $(p(x_s)-p(x_0))_{i_\alpha i_{\al + 1}}$. The expression for $C_{i_{-1}i_0 \ldots i_{M-1}}$ involves products of terms of the previous three forms, so that if one is zero, then $C_{i_{-1}i_0 \ldots i_{M-1}}=0$.

In the first case, since $i_{\al} \neq i_{\al + 1}$, $x_0 \in U$ and $\supp( \chi_{i_\al}) \cap \supp(\chi_{i_{\al + 1}}) = K_{i_\alpha i_{\al + 1}}$, we have that
\[
\gamma_p(\lambda,x_0)_{i_\alpha i_{\al + 1}} = \frac{\sqrt{\chi_{i_\al} \chi_{i_{\al + 1}}}(x_0)}{\lambda-1} - \frac{\sqrt{\chi_{i_\alpha}\chi_{i_{\al + 1}}}(x_0)}{\lambda} = 0 .
\]
The second case follow from the third case, which is obvious since $x_s, x_0 \in U$ so that
\[
(p(x_s)-p(x_0))_{i_\alpha i_{\al + 1}}
= \sqrt{\chi_{i_\alpha}\chi_{i_{\al + 1}}}(x_s)- \sqrt{\chi_{i_\alpha}\chi_{i_{\al + 1}}}(x_{0}) = 0 . \qedhere
\]
%
%
\end{proof}

\begin{proof}[Proof of Theorem \ref{thm:projid}]
When $n=0$, we have that 
\[
\Chern^0_\phi(p_A)(x)=\phi(p_A(x)) = \sum_{i=1}^N\chi_i(x)\phi(g_{ii}(x)) = \phi(q),
\]
for all $x \in X$. When $n>0$, we obtain from Lemma \ref{l:realtech} that
\[
\sum_{i=1}^N \phi\Big(\Big( e_{p_A}(x,t) d\big( e_{p_A}(x,t) \big)^{\we 2n} \Big)_{\hspace*{-0.4ex}ii}\Big)
 = \sum_{i=1}^N \text{Tr} \Big(e_p(x,t) d\big( e_p(x,t) \big)^{\we 2n}\Big)\cdot \phi(g_{ii}(x_0)),
\]
for all $x \in \G U^{2n + 1}$ and all $t \in \De^{2n}$. Hence, since $\phi(g_{ii}(x_0))= \phi(q)$ and the projection $p$ is Murray-von Neumann equivalent to $1\in C(X)$, we obtain that
\[
\Chern^{2n}_\phi([p_A])=\Chern^{2n}([1]) \cd \phi(q) = 0 . \qedhere
\]
\end{proof}

\section{Index theorem --- Proof of Theorem B}

In this section we prove a slight generalization of Atiyah's $L^2$-index theorem in a $K$-theoretic setting. The context will be as follows:

We consider a flat bundle over a second countable, compact Hausdorff space $X$. The model fiber of our bundle will be a fixed finitely generated projective module, $q A^m$, over a unital $C^*$-algebra $A$ ($q\in M_m(A)$. We assume that $A$ is equipped with a faithful tracial state $\phi : A \to \mathbb{C}$. 

Following the procedure described in the beginning of Subsection \ref{sec:flat}, the above data yields a projection
\[
p_A \in M_{N \cd m}( C(X,A)),
\]
given by the formula in Equation \eqref{eq:pproj}. In particular, we get a class $[p_A] \in K_0(C(X,A))$ and we may define the index homomorphism
\[
\T{ind}_A : \KK[0][]{C(X)}{\Cc} \to \rr
\]
as the composition
\[
\xymatrix@C=3cm{
\KK[0][]{C(X)}{\Cc} \ar[d]^-{\tau_{A}} & & \\
\KK[0][]{C(X, A)}{A}  \ar[r]^-{[p_{A}]\hatotimes_{C(X,A)}-} & K_0(A) \ar[r]^-{\phi_*} & \rr .
}
\]
On the other hand, we still have the simple index map 
\[
\T{ind} : \KK[0][]{C(X)}{\Cc} \to \zz
\]
obtained by pairing with the class $[1_{C(X)}]\in K_0(C(X))$.

The (generalized) $K$-theoretic version of Atiyah's $L^2$-index theorem can now be stated as follows (we provide a proof at the end of this section):

\begin{theorem}\label{thm:iindex}
We have the identity
\[
\T{ind}_{A} = \phi(q) \cd \T{ind} : \KK[0][]{C(X)}{\Cc} \to \rr.
\]
In particular, 
\[
\T{ind}_{A}(x) \in \phi(q) \cd \zz \q \mbox{for all } x \in \KK[0][]{C(X)}{\Cc} .
\]
\end{theorem}

\begin{corollary}\label{t:gindtriv}
Suppose that $G$ is a countable discrete group and let $\phi\colon C^*_r(G)\to \C$ denote the canonical faithful tracial state. Suppose that $p : \tilde{X} \to X$ is a principal $G$-bundle, where $X$ is a second-countable, compact Hausdorff space. Then
\[
\phi_*\circ \eta_{\tilde{X}} = \T{ind}_{C_r^*(G)} = \T{ind} : \KK[0][]{C(X)}{\Cc} \to \rr.
\]
In particular, 
\[
\T{ind}_{C_r^*(G)}(x) \in \zz \q \mbox{for all } x \in \KK[0][]{C(X)}{\Cc} .
\]
\end{corollary}
\begin{proof}
If $\{V_i\}_{i=1}^N$ is a finite open cover of $X$, such that $\{p^{-1}(V_i)\}_{i=1}^N$ is a trivializing cover for $p : \tilde{X} \to X$, then we get locally constant transition functions
\[
g_{ij}\colon V_i\cap V_j \to G,
\]
satisfying the cocycle relations $g_{ii}=e$ and $g_{ij}g_{jk}=g_{ik}$ whenever $V_i\cap V_j\cap V_k$ is non-empty. If we compose with the left regular representation, we get locally constant maps into the unitary group of $C^*_r(G)$
\[
\la_{g_{ij}}\colon V_i\cap V_j\to U(C^*_r(G)),
\]
which fit the setup outlined in Section \ref{sec:flat}. We now apply the previous theorem with $A=C^*_r(G)$ and $q=1_A$.
\end{proof}

To give the proof of Theorem \ref{thm:iindex}, we first recall how to embed $A$ into a $\T{II}_1$-factor:

\begin{proposition}\label{prop:icc}
There exists a von Neumann $\text{II}_1$-factor $B$ equipped with a faithful tracial state $\lambda : B \to \cc$ and an embedding $\iota\colon A \hookrightarrow B$ such that $\lambda(\iota(x))=\phi(x)$ for any $x\in A$.
\end{proposition}
\begin{proof}
We let $\mathbb{L}(A)$ be the finite von Neumann algebra obtained by closing (in the weak operator topology) the image of $A$ under the faithful representation coming from the GNS construction associated to the faithful tracial state $\phi : A \to \cc$.

We set $B = \mathbb{L}(A) * \mathbb{L}(\mathbb{F}_2)$, the free product of $\mathbb{L}(A)$ and the group von Neumann algebra of the free group on two generators. Since $\mathbb{L}(A)$ is finite and $\mathbb{L}(\mathbb{F}_2)$ is a $\text{II}_1$-factor, the required properties of $B$ follow from \cite[Theorem 2.1]{dyk:vnfree}.
\end{proof}

Set $p_B=\iota(p_A) \in M_{m \cd N}(C(X,B))$ and define the von Neumann algebraic index map
\[
\T{ind}_B =  \la_* \ci ( [p_B] \hot_{C(X,B)} -) \ci \tau_B : \KK[0][]{C(X)}{\Cc} \to \rr.
\]

The following two simple lemmas are also needed for the proof of Theorem \ref{thm:iindex}. 

\begin{lemma}\label{l:vnind}
We have the identity
\[
\T{ind}_A = \T{ind}_B : \KK[0][]{C(X)}{\Cc} \to \rr.
\]
\end{lemma}
\begin{proof}
We put $f_A =[p_{A}]\hatotimes_{C(X,A)}-$ and $f_B = [p_B] \hot_{C(X,B)} - $ and notice that each subdiagram in the following diagram is commutative:
\[
\xymatrix{\KK[0][]{C(X)}{\Cc}\ar[d]^-{\tau_{A}} \ar[r]^-{\tau_B} & \KK[0][]{C(X,B)}{B} \ar[d]^-{\iota^*} \ar[dr]^-{f_B} &\\
\KK[0][]{C(X,A)}{A}\ar[dr]_-{f_A \qquad}\ar[r]^-{\iota_*} & \KK[0][]{C(X,A)}{B}\ar[r]^-{f_A} & K_0(B)\ar[d]^-{\lambda_*}\\
& K_0(A) \ar[ur]^-{\iota_*} \ar[r]^-{\phi_*} & \R .}
\]
This proves the lemma.
\end{proof}

%

\begin{lemma}\label{l:trivial}
The product of $\phi(q)$ and the index map, $\phi(q) \cd \T{ind}\colon \KK[0][]{C(X)}{\Cc} \to \rr$ agrees with the composition
\[
\xymatrixcolsep{4pc}\xymatrix{
\KK[0][]{C(X)}{\Cc} \ar[r]^-{\tau_B} & \KK[0][]{C(X,B)}{B} \ar[d]^-{[i(q)]\hot_{C(X,B)}-} & \\
& \KK[0][]{\Cc}{B} \ar[r]^-{\lambda_*} & \rr ,
}
\]
where $i(q) \in M_m( C(X,A) ) \su M_m( C(X,B) )$ denotes the constant map equal to $q \in M_m(A) \su M_m(B)$ everywhere.
\end{lemma}
\begin{proof}
The result of the lemma follows by noting that the diagram here below is commutative
\[
\xymatrixcolsep{2.8pc}\xymatrix{\KK[0][]{C(X)}{\Cc}\ar[d]^-{\io^*} \ar[r]^-{\tau_B} & \KK[0][]{C(X,B)}{B} \ar[d]^-{\iota^*} \\ 
\KK[0][]{\Cc}{\Cc}\ar[d]_-{\cong}\ar[r]^-{\tau_B} & \KK[0][]{B}{B}\ar[dr]^-{[q] \hot_B - } \\
\Z \ar[r]_-{\phi(q) \cdot } & \R & \KK[0][]{\Cc}{B} \ar[l]^-{\lambda_*} , }
\]
where the right vertical $*$-homomorphism $i : N \to C(X,N)$ sends operators to constant maps. Indeed, the bottom diagram commutes since $([q] \hot_B - )\ci \tau_B$ maps the (positive) generator in $\KK[0]{\Cc}{\Cc}$ to the class in $K_0(B)$ represented by $q \in M_m(A) \su M_m(B)$. 
%
%
\end{proof}

%
%
%

\begin{proof}[{\bf Proof of Theorem \ref{thm:iindex}}]
By Proposition \ref{prop:icc} we have a trace-preserving embedding of $A$ in $B$ such that the von Neumann algebra $B$ is a $\T{II}_1$-factor. By Proposition \ref{prop:injchern} and Theorem \ref{thm:projid} we have that $[p_B] = [i(q)] \in \KK[0][]{\cc}{C(X,B)}$ and thus by Lemma \ref{l:vnind} and Lemma \ref{l:trivial} that
\[
\begin{split}
\T{ind}_{A}(x) & = \T{ind}_{B}(x) 
= \lambda_*\big( [ p_{B}] \hot_{C(X,B)} \tau_{B}(x) \big) \\
& = \lambda_*\big( [i(q)] \hot_{C(X,B)} \tau_{B}(x)\big)
= \phi(q)\cdot \T{ind}(x).
\end{split}
\]
This proves the theorem.
\end{proof}



\bibliography{Bibliography}
\bibliographystyle{amsplain-init}

\end{document}